\documentclass[12pt]{amsart}

\usepackage{latexsym, amssymb, amscd, amsfonts,xypic,pstricks,enumitem,amsmath}

\usepackage{microtype}
\usepackage{ifthen}
\usepackage{longtable}

\newboolean{showpicture}
\setboolean{showpicture}{true}

\newboolean{bourbaki}
\setboolean{bourbaki}{true}

\newboolean{draft}
\setboolean{draft}{false}

\newboolean{shproofs}
\setboolean{shproofs}{true}

\newcommand{\np}{\medskip}

\numberwithin{equation}{section}

\renewcommand{\geq}{\geqslant}
\renewcommand{\leq}{\leqslant}

\newcommand{\Osh}{{\mathcal O}}                        

\renewcommand{\H}{\mathrm{H}}                          

\newcommand{\Pic}{\operatorname{Pic}} 

\ifthenelse{\boolean{bourbaki}}
{
\newcommand{\PP}{\mathbf{P}} 
}
{
\newcommand{\PP}{\mathbb{P}} 
}

\ifthenelse{\boolean{draft}}
{

\newcommand{\remind}[1]{{\bf[#1]}}
\newcommand{\lremind}[1]{{\bf[label:  #1]}}
\newcommand{\bremind}[1]{{\bf[label:  #1]}}

\newcommand{\comment}[1]{{\bf [#1]}}
}
{

\newcommand{\remind}[1]{{}}
\newcommand{\lremind}[1]{{}}
\newcommand{\bremind}[1]{{}}

\newcommand{\comment}[1]{{}}
}

\newcommand{\hiddenproof}[1]{
\ifthenelse{\boolean{shproofs}}{
\medskip
\begin{centering}
\begin{minipage}{0.9\textwidth}
\hrule
\vspace{.25\baselineskip}
\small
#1
\vspace{0.25\baselineskip}
\hrule
\end{minipage}\\
\end{centering}
\medskip
}
{
}
}

\newtheorem{theorem}{Theorem}[section]
\newtheorem{lemma}[theorem]{Lemma}
\newtheorem{corollary}[theorem]{Corollary}
\newtheorem{proposition}[theorem]{Proposition}

\begin{document}
\title{Segre embeddings and the canonical image of a curve}

\author{Nathan Grieve}
\address{Department of Mathematics and Statistics, Queen's University,
Kingston, ON, K7L 3N6, Canada}
\email{nathangrieve@mast.queensu.ca}

\maketitle
\begin{abstract}
We prove that there is no $g$ for which the canonical embedding of a general curve of genus $g$ lies on the Segre embedding of any product of three or more projective spaces.
\end{abstract}
\section{Introduction}

If $g$ is composite then the canonical embedding of a general curve of genus $g$ lies on the Segre embedding of a product of two projective spaces.  For example, a general curve of genus $4$ lies on the Segre embedding of ${\PP}^1\times {\PP}^1$ while a general curve of genus $6$ lies on the Segre embedding of ${\PP}^1\times {\PP}^2$.  These facts have applications concerning the structure Chow ring of $\mathcal{M}_g$ as illustrated in \cite[p. 421]{Faber:chowII} and \cite[p. 26]{Penev}.  The aim of this note is to prove that, by contrast, there is no $g$ for which the canonical embedding of a general curve of genus $g$ lies on the Segre embedding of any product of three or more projective spaces.  

To prove our main result we first give the following criterion for a general curve to lie on some Segre embedding ${\PP}^{r_1}\times \dots \times {\PP}^{r_n} \rightarrow {\PP}^{g-1}$.

\begin{proposition}\label{aa} The canonical image of a general curve $C$ of genus $g$ lies on some Segre embedding ${\PP}^{r_1}\times \dots \times {\PP}^{r_n}\rightarrow {\PP}^{g-1}$ if and only if $C$ admits line bundles $L_1,\dots, L_n$ such that $\bigotimes\limits_{i=1}^n L_i=K_{C}$, $h^0(C,L_i)=r_i+1$ and $\prod\limits_{i=1}^n (r_i+1)=g$.
\end{proposition} 
We then prove the following stronger result of independent interest.

\begin{theorem}\label{main:prop} Let $C$ be a general curve of genus $g$ and let $n\geq 3$.  If $C$ admits line bundles $L_1,\dots, L_n$ such that $\sum\limits_{i=1}^n \deg L_i=2g-2$ and $h^0(C,L_i)=r_i+1\geq 2$ for $i=1,\dots, n$,  then $n=3$ and 
$$\prod_{i=1}^3(r_i+1)<\biggl(\prod_{i=1}^3(r_i+1)\biggr)\left(\frac{r_1+r_2+r_3+2}{r_1+r_2+r_3+2-r_1r_2r_3}\right)\leq g.$$
Moreover, up to permutation of the $r_i$, we have
\begin{enumerate}
\item{$r_1=r_2=1$ and $1\leq r_3\leq \frac{g}{4}-2$}
or
\item{$r_1=1$, $r_2=2$ and $r_3=2,3,$ or $4$.}
\end{enumerate}
\end{theorem}
A more conceptual interpretation of the theorem is the following corollary.

\begin{corollary}\label{main:theorem}
Let $C$ be a general curve of genus $g$.  
\begin{enumerate}
\item{If $g$ is composite then the canonical image of $C$ lies on some Segre embedding $${\PP}^{r_1}\times {\PP}^{r_2}\rightarrow  {\PP}^{g-1}.$$ }
\item{If $n\geq 3$ then the canonical image of $C$ does not lie on any Segre embedding $${\PP}^{r_1}\times \dots \times {\PP}^{r_n}\rightarrow {\PP}^{g-1}.$$}
\end{enumerate}
\end{corollary}

Statement (a) in Corollary \ref{main:theorem} is well-known to experts and follows from Brill-Noether theory. We include a proof for completeness.  We refer to \cite[Chap. IV]{ACGH} and \cite[Chap. XXI]{ACG} for the basic results in Brill-Noether theory and \cite[p. xv]{ACGH} for notation.  For the non-expert we give a quick exposition of some of the main results of the theory in Section \ref{background}.  Throughout we assume that $g$ is an integer greater than or equal to $3$.

\subsection*{Acknowledgments}
I would like to thank my PhD advisor Mike Roth for several useful discussions, for encouraging me to write this note and for reading a preliminary draft.  I also benifited from conversations with Greg Smith and was partially supported by an Ontario Graduate Scholarship. 

\section{Brill-Noether theory}\label{background}
\np
For all smooth curves $C$ of genus $g$ there exist moduli schemes $W^r_d(C)$ whose closed points consist of equivalence classes of degree $d$ line bundles on $C$ with at least $r+1$ global sections.  Explicitly the closed points of $W^r_d(C)$  are identified with the set 
$$\{L \mid L\in \Pic(C), \deg L=d \textrm{ and } h^0(C,L)\geq r+1\}. $$  For the construction of $W^r_d(C)$ see \cite[p. 279]{Fulton:Lazarsfeld:1981} or \cite[p. 176]{ACGH}.

\np
Brill-Noether theory studies the geometry of such $W^r_d(C)$.  Central to the theory is the Brill-Noether number and the Petri-map.  The Brill-Noether number is defined by
$$ \rho(g,r,d):=g-(r+1)(g+r-d).$$
See \cite[p. 159]{ACGH} for an explanation as to how this number arises.  On the other hand, the Petri-map is defined for all line bundles $L$ on $C$.  If $K_C$ denotes the canonical bundle of $C$ then the Petri-map is the cup-product
$$ \mu(L): \H^0(C,L)\otimes \H^0(C,L^{\vee}\otimes K_C)\rightarrow \H^0(C,K_C).$$

\np
Both the Brill-Noether number and the Petri-map have remarkable geometric implications as seen in the two main theorems of Brill-Noether theory which we now describe.  The first theorem applies to all smooth curves and is the result of work by Kempf, Kleiman-Laksov and Fulton-Lazarsfeld.  

\begin{theorem}[Brill-Noether theory theorem I \cite{Kleiman:Laksov:1972}, \cite{Fulton:Lazarsfeld:1981}]\label{theoremA}
Let $C$ be a smooth curve of genus $g$.  
\begin{enumerate} 
\item{If $\rho(g,r,d)\geq 0$ then $W^r_d(C)\not = \varnothing$ and every irreducible component of $W^r_d(C)$ has dimension greater than or equal to $\rho(g,r,d)$.}
\item{If $\rho(g,r,d)\geq 1$ then $W^r_d(C)$ is connected.}
\end{enumerate}
\end{theorem}

If $C$ is a general curve then the converse to Theorem \ref{theoremA} (a) holds -- we state this in the second main theorem (Theorem \ref{theoremB}).  That the converse to Theorem \ref{theoremA} (a) holds for general curves was first proved by Griffiths-Harris \cite{Griffiths:Harris:1980} using a degeneration argument.

Petri conjectured that $\mu(L)$ is injective for all line bundles $L$ on a general curve $C$. (This is implicit in \cite{Petri}.  See the footnote on \cite[p. 215]{ACGH} for a discussion.)  Arbarello-Cornalba clarified the geometric implications of this conjecture.

\begin{theorem}[Arbarello-Cornalba theorem {\cite[Theorem 0.3]{Arbarello:Cornalba:1981}}]
Let $C$ be a general curve of genus $g$.  If $\mu(L)$ is injective for all line bundles $L$ on $C$ and $W^r_d(C)\not = \varnothing$ then every irreducible component of $W^r_d(C)$ has dimension $\rho(g,r,d)$ and $W^r_d(C)$ is non-singular away from $W^{r+1}_d(C)$.  
\end{theorem}

In \cite{Gieseker:1982} Gieseker used a degeneration argument, and which was subsequently streamlined by Eisenbud-Harris \cite{Eisenbud:Harris:1983}, to prove that $\mu(L)$ is injective for all line bundles on a general curve.  Lazarsfeld, without using degenerations, also gave an independent proof  \cite[p. 299]{Lazarsfeld:1986}. The fact that $\mu(L)$ is injective for all line bundles on a general curve is sometimes referred to as the  Gieseker-Petri theorem.

\begin{theorem}[Gieseker-Petri theorem {\cite[Theorem 1.1, p. 251]{Gieseker:1982}}]\label{petri:conj}
If $C$ is a general curve of genus $g$ then the cup-product $\mu(L)$ is injective for all line bundles $L$ on $C$.
\end{theorem}  

The above discussion, combined with \cite[Corollary 2.4, p. 280]{Fulton:Lazarsfeld:1981}, is summarized in the second main theorem of Brill-Noether theory.  

\begin{theorem}[Brill-Noether theory theorem II \cite{Arbarello:Cornalba:1981}, \cite{Fulton:Lazarsfeld:1981}, \cite{Gieseker:1982},\cite{Griffiths:Harris:1980},  \cite{Lazarsfeld:1986}]\label{theoremB}
Let $C$ be a general curve of genus $g$.
\begin{enumerate}
\item{If $W^r_d(C)\not = \varnothing$ then $\rho(g,r,d)\geq 0$, $W^r_d(C)$ is of pure dimension $\rho(g,r,d)$ and $W^r_d(C)$ is non-singular away from $W^{r+1}_d(C)$.}
\item{If $\rho(g,r,d)\geq 1$ then $W^r_d(C)$ is irreducible.}
\end{enumerate}
\end{theorem}

The results of this note depend on Theorem \ref{theoremA} (a), Theorem \ref{petri:conj}, and Theorem \ref{theoremB} (a).

\section{The $n$-fold Petri-map and the proof of Proposition \ref{aa}}

Proposition \ref{aa} relies on the following lemma which implies that, for a general curve, the $n$-fold Petri-map is injective.   
\begin{lemma}\label{n:petri}
Let $C$ be a general curve of genus $g$.  For all line bundles $L_1,\dots, L_n$ on $C$ such that $\bigotimes\limits_{i=1}^n L_i\cong K_C$ the cup-product
$ \bigotimes\limits_{i=1}^n \H^0(C,L_i)\rightarrow \H^0(C,K_C)$ is injective.
\end{lemma}
\proof \np The case $n=1$ is trivial.  The case $n=2$ is the Gieseker-Petri theorem (Theorem \ref{petri:conj}).  Let $n\geq 3$.  Without loss of generality we may assume that $\H^0(C,L_i)\not = 0$, $i=1,\dots n$.

The cup-product factors
$$ \xymatrix{ \H^0(C,L_1)\otimes \dots \otimes\H^0(C,L_n)\ar[r] \ar[d] & \H^0(C,K_C) \\ \H^0(C,L_1\otimes L_2)\otimes \H^0(C,L_3)\otimes \cdots \otimes \H^0(C,L_n)\ar[ru] &}. $$  By induction the diagonal arrow is injective. It thus suffices to show that the downward arrow is injective.  For this we reduce to showing that the cup-product
$$ \Phi: \H^0(C,L_1)\otimes \H^0(C,L_2)\rightarrow  \H^0(C,L_1\otimes L_2)$$ 
is injective.

By assumption there exist non-zero sections $\sigma_i\in \H^0(C,L_i), i=3,\dots n$.  These produce (via cup-product) a non-zero section $$\sigma=\sigma_3 \cdots  \sigma_n \in \H^0(C,L_3\otimes \cdots \otimes L_n).$$    Since $\sigma\not = 0$ multiplication by $\sigma$ yields the following two injections 
$$\H^0(C,L_2)\rightarrow \H^0(C,L_2\otimes \cdots \otimes L_n)$$
and 
$$\H^0(C,L_1\otimes L_2)\rightarrow \H^0(C,K_C).$$ 
Using the above we produce (by cup-product) the commutative diagram 
$$\xymatrix{\H^0(C,L_1)\otimes \H^0(C,L_2)\ar[r]^-{\Phi} \ar[d]_-{\operatorname{id}\otimes \cdot \sigma} & \H^0(C,L_1\otimes L_2) \ar[d]^-{\cdot \sigma} \\ \H^0(C,L_1)\otimes \H^0(C,L_2\otimes \dots \otimes L_n)\ar[r]^-{\mu(L_1)} & \H^0(C,K_C)}.$$ \\
The vertical arrows of the above diagram are injective, as just noted, whereas bottom arrow of the diagram is injective by the Gieseker-Petri theorem.  Hence the top arrow $\Phi$ is injective. \endproof 
We now use Lemma \ref{n:petri} to prove Proposition \ref{aa}.
\begin{proof}[Proof of Proposition \ref{aa}]
Let $\eta:C\rightarrow {\PP}^{g-1}$ be the canonical map.  If $C$ is contained in the image of some Segre  embedding $\phi:{\PP}^{r_1}\times \dots \times {\PP}^{r_n}\rightarrow {\PP}^{g-1}$ then there exists a closed immersion $\psi: C\rightarrow {\PP}^{r_1}\times \dots \times {\PP}^{r_n}$ making the diagram 
$$ \xymatrix{C\ar[d]_-{\psi} \ar[r]^-{\eta} & {\PP}^{g-1} \\ {\PP}^{r_1}\times \dots \times {\PP}^{r_n}\ar[ur]_-{\phi}}$$ commute.

For every $1\leq i \leq n$, let $\pi_i$ denote the projection of ${\PP}^{r_1}\times \cdots \times {\PP}^{r_n}$ onto the $i$-th factor and set $L_i:=(\pi_i\circ \psi)^*\Osh_{{\PP}^{r_i}}(1)\cong\psi^*(\pi_i^* \Osh_{{\PP}^{r_i}}(1))$. 
We then obtain
$$K_C\cong\eta^*\Osh_{{\PP}^{g-1}}(1)\cong(\phi\circ \psi)^*\Osh_{{\PP}^{g-1}}(1)\cong L_1\otimes \cdots \otimes L_n.$$
Since the canonical image of $C$ is non-degenerate we conclude that $h^0(C,L_i)\geq r_i+1$ for $i=1,\dots, n$.

Finally, since $\eta$ is induced by the complete canonical series, we conclude that the cup-product $\bigotimes\limits_{i=1}^n \H^0(C,L_i)\rightarrow \H^0(C,K_C) $ is surjective.   By Lemma \ref{n:petri} the cup-product is injective and, by assumption, $\prod\limits_{i=1}^n (r_i+1)=g$. It follows that $h^0(C,L_i)=r_i+1$ for $i=1,\dots, n$ .

Conversely, given such $L_1,\dots, L_n$ we get regular maps $C\rightarrow{\PP}^{r_i}$ for $i=1,\dots, n.$  We thus can make a regular map $\eta:C\rightarrow {\PP}^{g-1}$ (induced by a (sub)-canonical series) by composition $\xymatrix{C\ar[r] & {\PP}^{r_1}\times \dots \times {\PP}^{r_n}\ar[r]^-{\textrm{Segre}} & {\PP}^{g-1}}.$  By Lemma \ref{n:petri} the cup-product \\ $\bigotimes\limits_{i=1}^n \H^0(C,L_i)\rightarrow \H^0(C,K_C)$ is injective. Since $\prod\limits_{i=1}^n(r_i+1)=g$ it is also surjective.  We thus conclude that the resulting map $\eta$ is given by the complete canonical series.  \end{proof}

\section{Proof of main Theorem and its corollary }

We first prove Corollary \ref{main:theorem} (a).  We then prove Theorem \ref{main:prop} from which we deduce Corollary \ref{main:theorem} (b).

\subsection*{The case $n=2$ and $g$ is composite}
When $n=2$ and $g$ is composite it is easy to prove that, in its canonical embedding, a general curve of genus $g$ lies on the image of some (non-trivial) Segre embedding  ${\PP}^{r_1}\times {\PP}^{r_2}\rightarrow  {\PP}^{g-1}.$ 

\begin{proof}[Proof of Corollary \ref{main:theorem} (a)]
Since $g$ is composite we can write $g=(r_1+1)(r_2+1)$ with $r_i\geq 1$.  Set $d_1=\frac{r_1g}{r_1+1}+r_1=r_1r_2+2r_1$.  Then $\rho(g,r_1,d_1)=0$ so, by Theorem \ref{theoremA} (a), $C$ admits a line bundle $L_1$ with at least $r_1+1$ global sections. On the other hand $\rho(g,r_1+1,d_1)<0$, so Theorem \ref{theoremB} (a) implies that $W^{r_1+1}_{d_1}=\varnothing$. Hence $L_1$ has exactly $r_1+1$ global sections. Set $L_2:=L_1^\vee \otimes K_C$.  Then, by the Riemann-Roch theorem, we obtain
$h^0(C,L_2)=g-d_1+r_1.$
Simplifying and using our expressions for $g$ and $d_1$ above we obtain \\
$h^0(C,L_2)=(r_1+1)(r_2+1)-r_1r_2-2r_1+r_1=r_2+1$.   The assertion now follows from Proposition \ref{aa}. 
\end{proof}

\subsection*{The case $n\geq 3$}

When $n\geq 3$ and $r_i\geq 1$, for $i=1,\dots, n$, the situation is in stark contrast to the case $n=2$.  Indeed we prove that there is no $g$ for which the canonical embedding of a general curve lies on the Segre embedding of any product of three or more projective spaces.  We deduce this result from Theorem \ref{main:prop} whose proof occupies the rest of this section.  The theorem follows from the following more general observation.

\begin{proposition}\label{aprop}
Let $C$ be a general curve of genus $g$.  Let $d$ be a non-negative integer and let $q:=\Big \lfloor \frac{d}{g} \Big \rfloor$.  If $C$ admits line bundles $L_1,\dots, L_n$ such that $r_i+1=h^0(C,L_i)\geq 2$ and $\deg \bigotimes\limits_{i=1}^n L_i=d$, then $n<2q+2$ and $g\geq\frac{1+\left(\sum\limits_{i=1}^n r_i \right)}{ -n + (q+1)+\left(\sum\limits_{i=1}^n \frac{1}{r_i+1}\right)}$.
\end{proposition}
\begin{proof}
Since $C$ is general, if such $L_i$ exist then, by Theorem \ref{theoremB} (a), $$ \rho(g,r_i,d_i)=g-(r_i+1)(g+r_i -d_i)\geq 0 \textrm{ for $i=1,\dots, n$}.$$
Solving for $d_i$ we conclude $d_i\geq g+r_i - \frac{g}{r_i+1} \textrm{ for } i = 1,\dots, n.$
Let $r$ denote the remainder obtained by dividing $d$ by $g$.  Then $0\leq r < g$ and 
$$d=(q+1)g+r-g=\sum\limits_{i=1}^n d_i \geq ng + \left(\sum\limits_{i=1}^n r_i\right)-\left(\sum_{i=1}^n \frac{g}{r_i+1}\right).$$  Rearranging we obtain
\begin{equation}\label{AA} \left(n-(q+1)-\left(\sum\limits_{i=1}^n \frac{1}{r_i+1}\right)\right)g \leq r-g-\left(\sum_{i=1}^n r_i\right)<0. \end{equation}  \\ Since $g>0$ we conclude 
\begin{equation}\label{BB}  n-(q+1)-\left(\sum\limits_{i=1}^n \frac{1}{r_i+1}\right) <0.\end{equation}  Now by assumption $r_i\geq 1$, for $i=1,\dots, n$.  Thus \begin{equation}\label{CC} \sum\limits_{i=1}^n \frac{1}{r_i+1} \leq \frac{n}{2}.\end{equation}  Using equations \eqref{CC} and \eqref{BB} we deduce that $n<2q+2$.  Finally, dividing equation \eqref{AA} by equation \eqref{BB} we obtain
\begin{equation}\label{DD} g \geq \frac{-r+g+\left(\sum\limits_{i=1}^n r_i \right)}{ -n + (q+1)+\left(\sum\limits_{i=1}^n \frac{1}{r_i+1}\right)} \geq \frac{1+\left(\sum\limits_{i=1}^n r_i \right)}{ -n + (q+1)+\left(\sum\limits_{i=1}^n \frac{1}{r_i+1}\right)}. \end{equation}  \end{proof}
 We now use Proposition \ref{aprop} to prove Theorem \ref{main:prop}.
 \begin{proof}[Proof of Theorem \ref{main:prop}]
 Set $d_i=\deg L_i$, $i=1,\dots, n$ and set $d=2g-2$.  Then $q:=\Big \lfloor \frac{d}{g} \Big \rfloor=1$ and the remainder $r$ equals $g-2$.  Since $r_i\geq 1$, for $i=1,\dots, n$, applying Proposition \ref{aprop} we conclude that $n<4$.  Since $n\geq 3$ we conclude that $n=3$.  Substituting $n=3$ and $r=g-2$ into equation \eqref{DD} we obtain $$g\geq \frac{2+\sum\limits_{i=1}^3 r_i}{-1 + \sum\limits_{i=1}^3 \frac{1}{r_i+1} }.$$
Rearranging we obtain
$$ g\geq \left( \prod\limits_{i=1}^3(r_i+1) \right) \left( \frac{r_1+r_2+r_3+2}{r_1+r_2+r_3+2-r_1r_2r_3}\right)>\prod\limits_{i=1}^3 (r_i+1).$$ \\
Since $1<\sum\limits_{i=1}^3 \frac{1}{r_i+1}$  we conclude, up to permutation of the $r_i$, that $r_1=r_2=1$, $r_3\geq 1$ or $r_1=1,r_2=2$ and $2\leq r_3\leq 4$.  Finally if $r_1=r_2=1$ and $r_3\geq 1$ then the condition $\prod\limits_{i=1}^3 (r_i+1)<g$ implies that $1\leq r_3\leq \frac{g}{4}-2$. \end{proof}
Having proved Theorem \ref{main:prop}, we now deduce Corollary \ref{main:theorem} (b).
\begin{proof}[Proof of Corollary \ref{main:theorem} (b).]  
Let $C$ be a general curve of genus $g$.  Suppose that $C$ lies on the image of some Segre embedding
${\PP}^{r_1}\times \dots \times {\PP}^{r_n}\rightarrow {\PP}^{g-1}.$  By Proposition \ref{aa}, $C$ admits line bundles $L_1,\dots, L_n$ such that $\bigotimes\limits_{i=1}^gL_i=K_C$, $h^0(C,L_i)=r_i+1$ and $\prod\limits_{i=1}^n(r_i+1)=g$. By Theorem \ref{main:prop}, $n=3$ and $g>\prod\limits_{i=1}^3 (r_i+1)$. This is a contradiction.  \end{proof}

%

\begin{thebibliography}{}
%
%
\bibitem[ACG]{ACG} Arbarello, E., Cornalba, M. and Griffiths, P. A. \textit{Geometry of algebraic curves. {V}ol. {II}} (Springer-Verlag, Berlin Heidelberg, 2011).
\bibitem[ACGH]{ACGH}Arbarello, E., Cornalba, M., Griffiths, P. A. and
              Harris, J., \textit{Geometry of algebraic curves. {V}ol. {I}} (Springer-Verlag, New York, 1985).
              
\bibitem[AC]{Arbarello:Cornalba:1981}Arbarello, E. and Cornalba, M.,  \textit{On a conjecture of {P}etri},  Comment. Math. Helv. \textbf{56} (1981) 1--38.

\bibitem[EH]{Eisenbud:Harris:1983}  Eisenbud, D. and Harris, J., \textit{A simpler proof of the {G}ieseker-{P}etri theorem on special divisors}, Invent. Math. \textbf{74} (1983) 269--280. 

\bibitem[Fab]{Faber:chowII} Faber, C., \textit{Chow rings of moduli spaces of curves. {II}. {S}ome results on
              the {C}how ring of $\overline{\mathcal{M}_4}$} Ann. of Math. (2) (1990) 421--449.
              
\bibitem[FL]{Fulton:Lazarsfeld:1981} Fulton, W. and Lazarsfeld, R., \textit{On the connectedness of degeneracy loci and special divisors}, Acta Math. \textbf{146} (1981) 271--283.

\bibitem[Gie]{Gieseker:1982} Gieseker, D., \textit{Stable curves and special divisors: {P}etri's conjecture}, Invent. Math.  \textbf{66} (1982) 251--275.

\bibitem[GH]{Griffiths:Harris:1980} Griffiths, P. and Harris, J., \textit{On the variety of special linear systems on a general
              algebraic curve},  Duke Math. J. \textbf{47} (1980) 233--272.
              
\bibitem[KL]{Kleiman:Laksov:1972}  Kleiman, S. L. and Laksov, D., \textit{On the existence of special divisors}, Amer. J. Math. \textbf{94} (1972) 431--436.

\bibitem[Laz]{Lazarsfeld:1986}  Lazarsfeld, R., \textit{Brill--{N}oether--{P}etri without degenerations}, J. Differential Geom.  \textbf{23}  (1986)  299--307.       

\bibitem[Pen]{Penev}  Penev, N.  \textit{The Chow ring of $\mathcal{M}_6$}, PhD thesis, Standford University (2009) pp.46.

\bibitem[Pet]{Petri} Petri, K. \textit{\"{U}ber {S}pezialkurven. {I}}, Math. Ann. \textbf{93} (1925)  182--209.
\end{thebibliography}
%

\end{document}